\documentclass{amsart}
\usepackage{graphicx}
\usepackage{ulem}
\usepackage{slashed}
\usepackage{tikz}
\usepackage{amssymb}

\newtheorem{theorem}{Theorem}

\newtheorem{theoremL}{Theorem}
\newtheorem{lemma}[theoremL]{Lemma}

\newtheorem{theoremCr}{Theorem}
\newtheorem{corollary}[theoremCr]{Corollary}

\numberwithin{equation}{section}

\begin{document}

\title{On the existence of twin prime in an interval}

\author{Shaon Sahoo}
\address{Indian Institute of Technology, Tirupati 517506, India}
\email{shaon@iittp.ac.in}


\begin{abstract}
Let $S_{(x,y]} = \left\{\frac{p_n}{p_{n+1}-2} :~ n\in I \right\}$, where $I = \left\{n :~ x<p_n \le y \right\}$,
$p_n$ is the $n$-th prime and $x, y \in \mathbb{R}_{>0}$. If $M_\alpha(x,y)$ denotes the $\alpha$-power mean of 
the elements of $S_{(x,y]}$, it is shown that the existence of a twin prime pair in $(x,y]$ is implied if 
$\displaystyle \lim_{\alpha \rightarrow \infty}M_{\alpha}(x,y) > 1 - 2/y + O(y^{-2})$ for a sufficiently large $y$.
For a special choice of $y$, we also find a lower bound for the mean: 
$\displaystyle \lim_{\alpha \rightarrow \infty}M_{\alpha}(x,x^\beta)>1-c/x^\beta+O(x^{-\beta}\log^{-1} x)$, 
where the constant $c>0$ and $\beta = 1+c/\log^2 x$ or equivalently, $x^\beta=x+cx/\log x+O(x/\log^2 x)$. 
With $c<2$, the lower bound for $\displaystyle \lim_{\alpha \rightarrow \infty}M_{\alpha}(x,x^\beta)$ satisfies the 
inequality on the existence of a twin prime in the interval $(x,x^\beta]$.
\end{abstract}

\maketitle

\section{Introduction}
\label{sec1}
Let $p_n$ be the $n$-th prime and $x, y \in \mathbb{R}_{>0}$ with $x<y$. Let 
$S_{(x,y]}=\{r_n :~ n\in I \}$, where $r_n = \frac{p_n}{p_{n+1}-2}$ and $I=\{n :~x<p_n \le y\}$. The power mean or H\"{o}lder mean
of the elements of $S_{(x,y]}$ is given by $M_{\alpha}(x,y) = (\frac{1}{\pi(x,y)} { \sum_{n\in I}} r_n^{\alpha})^{\frac{1}{\alpha}}$, 
where $\alpha \in \mathbb{R}$ and $\pi(x,y)$ denotes the number of primes in $(x,y]$. By construction, $r_n\le 1$ for all $n \in I$, and 
$r_n=1$ only when $p_{n+1}-p_n=2$. By Theorem \ref{th3}, $\displaystyle \lim_{\alpha \rightarrow \infty}M_{\alpha}(x,y)$ converges to the largest 
element in $S_{(x,y]}$. 
Therefore, $\displaystyle \lim_{\alpha \rightarrow \infty}M_{\alpha}(x,y)=1$ implies that there exists a twin 
prime in $(x,y]$. In fact, as stated in Theorem \ref{th2}, the existence of a twin prime pair in the interval $(x,y]$ is implied if 
$\displaystyle \lim_{\alpha \rightarrow \infty}M_{\alpha}(x,y) > 1 - 2/y + O(y^{-2})$ for a sufficiently large $y$. 

In this work we show that, for a special $x$-dependent choice of $y$, a lower bound for the mean $\displaystyle \lim_{\alpha \rightarrow \infty}M_{\alpha}(x,y)$ 
is actually larger than $1 - 2/y + O(y^{-2})$ for a sufficiently large $y$ (or $x$). The result on the lower bound is stated as the following theorem.  


\begin{theorem}
Let $S_{(x,y]} = \left\{\frac{p_n}{p_{n+1}-2} :~ n \in I \right\}$ for $\displaystyle I = \left\{n :~x<p_n \le y \right\}$,  where $p_n$ is the $n$-th prime.
With $r_n= \frac{p_n}{p_{n+1}-2}$, the power mean of the elements of $S_{(x,y]}$ is given by   
$\displaystyle M_{\alpha}(x,y)= \left(\frac{1}{\pi(x,y)}\sum_{n \in I} r_n ^{\alpha}\right)^\frac{1}{\alpha}$, 
where $\alpha$ is a real number and $\pi(x,y)$ denotes the number of primes in the interval $(x,y]$. We have
\begin{equation}
\lim_{\alpha \rightarrow \infty}M_{\alpha}(x,x^\beta)>1-\frac{c}{x^\beta}+O(\frac{1}{x^{\beta}\log x}), 
\label{eq1.1}
\end{equation}
where $c$ is a fixed positive constant ($c>0$) and $\beta = 1+c/\log^2 x$ 
or equivalently, $x^\beta=x+cx/\log x+O(x/\log^2 x)$.
\label{th1}
\end{theorem}

The proof of Theorem \ref{th1} is given in Section \ref{sec4}. 
If we take $c<2$ in Theorem \ref{th1}, the inequality in Theorem \ref{th2} is satisfied for the interval $(x,x^\beta]$:
${\displaystyle \lim_{\alpha \rightarrow \infty}}M_{\alpha}(x,x^\beta)>1-\frac{c}{x^\beta}+O(\frac{1}{x^{\beta}\log x})>
1 - \frac{2}{x^\beta} + O(\frac{1}{x^{2\beta}})$. This shows that there exists a twin prime in the interval $(x,x+\frac{cx}{\log x}+O(\frac{x}{\log^2 x})]$ for
sufficiently large $x$ and $0<c<2$. If a twin prime exists in an interval, it also exists in an extension of the interval. We, accordingly, conclude that a twin 
prime exists in the aforementioned interval for any fixed $c>0$.

\section{Background}
\label{sec2}
It is conjectured that there are infinite number of primes $p$, for each of which $p+2$ is also a prime. Many important contributions have been made in 
the efforts to prove the conjecture. In this regard we may note the Chen's work \cite{chen73} and more recent work by Zhang \cite{zhang14}. Chen showed 
that there are infinite number of primes $p$ so that $p+2$ has at most two prime factors. On the other hand, Zhang showed that there are infinite number 
of prime pairs with bounded gaps. More specifically, he showed that $\liminf_{n \rightarrow \infty} (p_{n+1}-p_n)<7\times 10^7$, where $p_n$ is the
$n$-th prime number. Subsequent works by Tao (through Polymath projects) and Maynard brought down the upper bound of 70 million 
to 246 \cite{tao14,maynard15}. It is still an open problem whether the upper bound of the limit can be brought down to the conjectured value of 2. 

We take here a different approach to study the question on the twin primes. 

\section{Some useful results}
\label{sec3}

\begin{theorem}
Let $S_{(x,y]}$ and $M_{\alpha}(x,y)$ respectively be the set and the power mean of the elements of the set, as defined in Theorem \ref{th1}. The existence
of a twin prime pair in the interval $(x,y]$ is implied if $\displaystyle \lim_{\alpha \rightarrow \infty}M_{\alpha}(x,y) > 1 - \frac 2y + O(\frac{1}{y^{2}})$ 
for a sufficiently large $y$.
\label{th2}
\end{theorem}

\begin{proof}
Taking $g_n= p_{n+1}-p_n$, we rewrite the elements of $S_{(x,y]}$ as $r_n=\frac{1}{1+(g_n-2)/p_n}$. We note that $r_n\le 1$; the equality sign holds only when 
$g_n=2$, i.e., when $p_n$ represents (the first member of) a twin prime. We, therefore, see that 1 is the largest possible element of $S_{(x,y]}$ and 
$1 \in S_{(x,y]}$ implies the existence of twin prime in  $(x,y]$.

Let $R$ be the largest element in $S_{(x,y]}$ which is not 1. 
A non-trivial upper bound of $R$ corresponds to $g_n=4$ and the 
largest prime in $(x,y]$. If $P$ is the largest prime in $(x,y]$, we have 
$R \le \frac{1}{1+(4-2)/P}\le \frac{1}{1+2/y}=1-2/y+O(y^{-2})$ for sufficiently large $y$. 
By Theorem \ref{th3}, $\displaystyle \lim_{\alpha \rightarrow \infty}M_{\alpha}(x,y)$ converges to the largest element in $S_{(x,y]}$. Therefore,
the inequality $\displaystyle \lim_{\alpha \rightarrow \infty}M_{\alpha}(x,y) > 1-2/y+O(y^{-2})$ implies that there is an element in $S_{(x,y]}$ which 
is larger than $R$. But 1 can be the only element which is larger than $R$. Hence  
$\displaystyle \lim_{\alpha \rightarrow \infty}M_{\alpha}(x,y) > 1-2/y+O(y^{-2})$ actually confirms that $1 \in S_{(x,y]}$. 
The statement of the lemma now follows.
\end{proof}

\begin{theorem}
Let $A=\{a_1, a_2, \cdots, a_N\}$ be a set of $N$ positive real numbers. Let $\max\{A\}$ and $\min\{A\}$ denote respectively the largest and the smallest elements 
in $A$. The $\alpha$-power mean of the elements of $A$ is given by 
$M_\alpha(A) = (\frac{1}{N} \sum_{i=1}^N a_i^{\alpha})^{\frac{1}{\alpha}}$, where $\alpha \in \mathbb{R}$. We have
\begin{enumerate}
\item $M_\alpha(A)\ge M_{\alpha'}(A)$ for $\alpha > \alpha'$; the equality sign holds iff all the elements in $A$ are equal.
\item ${\displaystyle \lim_{\alpha \rightarrow 0}M_{\alpha}(A)} =(\prod_{i=1}^N a_i)^{\frac 1N}$.
\item $\displaystyle \lim_{\alpha \rightarrow \infty}M_{\alpha}(A) = \max\{A\}$ and $\displaystyle \lim_{\alpha \rightarrow -\infty}M_{\alpha}(A) = \min\{A\}$.
\end{enumerate}
\label{th3}
\end{theorem}

These are standard results; the proof can be found elsewhere (eg. \cite{steele04}).

\begin{theorem}
We have,
\begin{equation}
\sum_{p\le x} \frac 1p = \log \log x + M + O(\exp(-\sqrt[14]{\log x})), \nonumber
\end{equation}
where $M$ ($\approx$ 0.261) is the Meissel-Mertens constant.
\label{th4}
\end{theorem}

This result (Theorem \ref{th4}) is due to Landau \cite{landau09}. Results with improved error term are known but are not necessary for our work here. 
A corollary of this result is the following. 

\begin{corollary}
We have,
\begin{equation}
\sum_{p\le x} \frac 1p = \log \log x + M + O(1/(\log x)^A), \nonumber
\end{equation}
for any constant $A>0$. 
\label{cr1}
\end{corollary}

For our work we need the above result with $A>1$; for definiteness, we will take $A=2$. 

\begin{lemma}
We have,
\begin{equation}
(1- \frac 12) \prod_{2<p\le x}(1-\frac{2}{p}) = \frac{e^{-D}}{\log^2 x}\left(1+O(\frac{1}{\log^2 x})\right), \nonumber 
\end{equation}
where $D$ ($\approx$ 0.877) is a constant.
\label{lm1}
\end{lemma}

\begin{proof}
Let ${\displaystyle G(x) = \prod_{2<p\le x}(1-\frac{2}{p})}$. Therefore,\\
${\displaystyle \log G_x = \sum_{2<p\le x} \ln (1-\frac 2p) = \sum_{2<p\le x}\left[ - \frac 2p  - \sum_{k=2}^{\infty} \frac 1k \left(\frac 2p\right)^k \right]}$.\\
Since $\sum_{k=2}^{\infty}\frac 1k \left(\frac 2p\right)^k \le \frac 12 \sum_{k=2}^{\infty} \left(\frac 2p\right)^k = \frac{2}{p^2(1-2/p)} \le \frac{6}{p^2}$,\\
the infinite sum ${\displaystyle \sum_{p>2} \sum_{k=2}^{\infty} \frac 1k \left(\frac 2p\right)^k }$ converges absolutely to a constant, say, $C$. By numerical
calculation we estimate that ${\displaystyle C = -\sum_{p>2}\left[\log (1 - \frac 2p)+ \frac 2p \right] \approx  0.660413}$. Moreover,
$\sum_{2<p\le x} \sum_{k=2}^{\infty} \frac 1k \left(\frac 2p\right)^k = \sum_{p>2} \sum_{k=2}^{\infty} \frac 1k \left(\frac 2p\right)^k - 
\sum_{p > x} \sum_{k=2}^{\infty} \frac 1k \left(\frac 2p\right)^k$. Here the first term is the constant $C$ and the second term 
$\sum_{p > x} \sum_{k=2}^{\infty} \frac 1k \left(\frac 2p\right)^k\le \sum_{p>x} \frac{6}{p^2} = O(\frac{1}{x})$. With this, we have\\
${\displaystyle \log G(x) = -2\sum_{2<p\le x} \frac 1p - C + O(\frac 1x)}$.\\
Using now the Corollary \ref{cr1} (with $A=2$), we get \\
${\displaystyle \log G(x) = -2\log\log x - 2M + 1 - C + O(\frac{1}{\log^2 x}) + O(\frac 1x)}$, or \\
${\displaystyle \log G(x) = -2\log\log x - D' + O(\frac{1}{\log^2 x})}$, where $D' = 2M+C -1 \approx 0.183407.$\\
This gives us ${\displaystyle G(x) = \frac{e^{-D'}}{\log^2 x}(1 + O(\frac{1}{\log^2 x}))}$.\\
Finally we get ${\displaystyle (1- \frac 12) \prod_{2<p\le x}(1-\frac{2}{p}) = \frac 12 G_x = \frac{e^{-D}}{\log^2 x}(1 + O(\frac{1}{\log^2 x}))}$,\\ 
where $D = D' + \log 2 \approx 0.876554$.
\end{proof}

\begin{lemma}
Let $\displaystyle T(x,x^\beta)=\prod_{n\in I} \frac{p_n}{p_{n+1}-2}$, where $p_n$ is the $n$-th prime and $I = \left\{n :~x<p_n \le x^\beta \right\}$.
Here $\beta>1$ and additionally, let $\beta$ also depend on $x$ in the following way: $\beta = O(1)$. We have 
$T(x,x^\beta)= \frac{\beta^2}{x^{\beta-1}}\left(1+O(\frac{1}{\log^2 x})\right).$ 
\label{lm2}
\end{lemma}

\begin{proof}
Let $p_s$ be the smallest prime greater than $x$ and similarly, $p_e$ be the smallest prime greater than $x^\beta$.

Now, $\displaystyle T(x,x^\beta)=\prod_{n\in I} \frac{p_n}{p_{n+1}-2} 
= \left(\frac{p_s-2}{p_e-2}\right)\frac{(1-\frac 12)\prod_{2<p\le x}(1-\frac 2p)}{(1-\frac 12)\prod_{2<p\le x^\beta}(1-\frac 2p)}$. 
Now using Lemma \ref{lm1}, we get, 
\begin{equation}
\begin{split}
T(x,x^\beta)=&\left(\frac{p_s-2}{p_e-2}\right)\frac{e^{-D}\log^{-2}x(1+O(\log^{-2}x))}{e^{-D}\log^{-2}x^\beta(1+O(\log^{-2}x^\beta))} \\
=&\beta^2 \frac{p_s-2}{p_e-2} \left(1+O(\frac{1}{\log^2 x})\right).
\end{split}
\label{eq3.1}
\end{equation}

We know from the prime number theorem that the average gap between the consecutive primes upto $x$ is $(1+o(1))\log x$. 
In fact it is conjectured, by Cram\'{e}r \cite{cramer36} and later with some refinement by Granville \cite{granville95}, that, if $G(x)$ is the largest gap 
between two consecutive primes upto $x$, then $G(x) = O(\log^2 x)$. 
Therefore, one would expect that $p_s = x + O(\log^t x)$ and $p_e = x^\beta + O(\log^t x)$, for some $t\ge 1$.

In this regard, 
currently the best unconditional result is due to Baker, Harman and Pintz \cite{baker01}. They showed that $G(x) = O(x^\theta)$, where $\theta = 0.525$. 
This result is sufficient to prove the lemma here. Taking $p_s = x + O(x^\theta)$ and $p_e = x^\beta + O(x^{\beta \theta})$ in Equation \ref{eq3.1}, we get,
\begin{equation}
\begin{split}
T(x,x^\beta)=&\beta^2\frac{x(1+O(\frac{1}{x^{1-\theta}}))}{x^\beta (1+O(\frac{1}{x^{\beta-\beta\theta}}))}\left(1+O(\frac{1}{\log^2 x})\right)\\ \nonumber
=&\frac{\beta^2}{x^{\beta-1}}\left(1+O(\frac{1}{\log^2 x})\right).
\end{split}
\end{equation}

\end{proof}
\section{Proof of Theorem \ref{th1}}
\label{sec4}

\begin{proof}
By Theorem \ref{th3}, $M_{\alpha}(x,x^\beta)\ge M_{\alpha'}(x,x^\beta)$ for $\alpha>\alpha'$. For the time being, we take $\beta>1$; 
its appropriate form will be mentioned soon. We are here, in particular, interested in the inequality  $M_{\infty}(x,x^\beta)\ge M_{0}(x,x^\beta)$, 
where $M_{\infty}(x,x^\beta)$ and
$M_{0}(x,x^\beta)$ respectively denote $\displaystyle \lim_{\alpha \rightarrow \infty}M_{\alpha}(x,x^\beta)$ and 
$\displaystyle \lim_{\alpha \rightarrow 0}M_{\alpha}(x,x^\beta)$. In the following we first argue that a strict inequality holds between the two means, i.e., 
$M_{\infty}(x,x^\beta)> M_{0}(x,x^\beta)$, and then we show that $M_{0}(x,x^\beta)=1-c/x^\beta+O(x^{-\beta}\log^{-1}x)$, where $c>0$ and
$\beta = 1+c/\log^2 x$ 
or equivalently, $x^\beta=x+cx/\log x+O(x/\log^2 x)$.

For the first part, we now note that $p_n\le p_{n+1}-2$ and, by Bertrand-Chebyshev theorem, $2p_n > p_{n+1}>p_{n+1}-2$. Therefore, $p_n\le p_{n+1}-2<2p_n$. This 
implies that $gcd(p_n,p_{n+1}-2)=1$, except when $p_n=p_{n+1}-2$. This observation helps us to determine if two elements of the set $S_{(x,x^\beta]}$ are equal or
not. Since the elements of $S_{(x,x^\beta]}$ are in the form $\frac{p_n}{p_{n+1}-2}$ ($=r_n$), the numerators of all the 
elements are different, except when the elements correspond to twin primes. Equivalently, we have $r_n \ne r_m$ whenever $n \ne m$, with an exception where both 
$p_n$ and $p_m$ are (first members of) twin primes. Since all the consecutive primes do not form twin prime pairs, we conclude 
that all elements in $S_{(x,x^\beta]}$ are not equal, and hence by Theorem \ref{th3}, 
$M_{\infty}(x,x^\beta)> M_{0}(x,x^\beta)$.

Next we analyze the mean $M_{0}(x,x^\beta)$. Let $\pi(x,x^\beta)$ denote the number of primes in $(x,x^\beta]$. 
If $\pi(x)$ denotes the number of primes upto $x$, we have $\pi(x) = \frac{x}{\log x}(1+O(\frac{1}{\log x}))$. 
Therefore, $\pi(x,x^\beta) = \pi(x^\beta)-\pi(x)=\frac 1\beta \frac{x^\beta}{\log x}(1+O(\frac{1}{\log x}))$. Using Lemma \ref{lm2}, we now have
\begin{equation}
\begin{split}
M_{0}(x,x^\beta) & = T(x,x^\beta)^{\frac{1}{\pi(x,x^\beta)}}\\
                 & =\left(\frac{\beta^2}{x^{\beta-1}}\right)^{\frac{1}{\pi(x,x^\beta)}}\left(1+O\left(\frac{1}{\pi(x,x^\beta)\log^2 x}\right)\right)\\
                 & = \left(\frac{\beta^2}{x^{\beta-1}}\right)^{\frac{1}{\pi(x,x^\beta)}}\left(1+O(x^{-\beta}\log^{-1} x)\right).
\end{split}
\label{eq4.1}
\end{equation}
Now to determine the main term, let $z=\left(\frac{\beta^2}{x^{\beta-1}}\right)^{\frac{1}{\pi(x,x^\beta)}}$. Taking logarithm on both sides, we have
\begin{equation}
\begin{split}
\log z & = \frac{2\log \beta - (\beta-1)\log x}{\pi(x,x^\beta)}\\
       & = \frac{2 \beta \log \beta \log x - \beta (\beta -1) \log^2 x}{x^\beta}\left(1+O(\frac{1}{\log x})\right).
\end{split}
\label{eq4.2}
\end{equation}
At this stage, we take $\beta = 1+ \frac{c}{\log^2 x}$, for a fixed $c>0$. Accordingly, we have $\beta \log \beta = \frac{c}{\log^2 x} + O(\frac{1}{\log^4 x})$, 
$\beta(\beta-1) = \frac{c}{\log^2 x}+O(\frac{1}{\log^4 x})$ and $x^\beta = x (x)^{c/\log^2 x} = x(1+\frac{c}{\log x}+O(\frac{1}{\log^2 x}))$. 
We now get from Equation \ref{eq4.2},
\begin{equation}
\begin{split}
\log z & = \left( \frac{2c}{x^\beta \log x}-\frac{c}{x^\beta}+O(\frac{1}{x^\beta \log^2 x})\right)  \left(1+O(\frac{1}{\log x})\right)\\
       & = -\frac{c}{x^\beta}+O(\frac{1}{x^\beta \log x}).
\end{split}
\label{eq4.3}
\end{equation}

This gives us $z = 1-\frac{c}{x^\beta}+O(\frac{1}{x^{\beta}\log x})$. Plugging this value of $z$ in Equation \ref{eq4.1}, we get
$M_{0}(x,x^\beta) = \left(1-\frac{c}{x^\beta}+O(\frac{1}{x^{\beta}\log x})\right)\left(1+O(\frac{1}{x^{\beta}\log x})\right) 
= 1-\frac{c}{x^\beta}+O(\frac{1}{x^{\beta}\log x})$. 

Finally we get $M_{\infty}(x,x^\beta)>M_{0}(x,x^\beta)=1-\frac{c}{x^\beta}+O(\frac{1}{x^{\beta}\log x})$.
\end{proof}


\bibliographystyle{amsplain}

\end{document}